\documentclass[10pt,leqno]{amsart}

\usepackage{graphicx}
\usepackage{amssymb,amsmath}%

\usepackage{subcaption}
\usepackage{float}
\usepackage{setspace}
\usepackage{array}
\usepackage{indentfirst,csquotes}

\usepackage{amssymb,amsthm,amsmath}
\usepackage{xcolor,paralist,hyperref,fancyhdr,etoolbox}
\numberwithin{equation}{section}
\newtheorem{theorem}{Theorem}[]

\newtheorem{example}[theorem]{Example}
\newtheorem{lemma}[theorem]{Lemma}

\newtheorem{corollary}[theorem]{Corollary}

\newcommand{\dx}{\,{\rm d}_F^\alpha x}

\hypersetup{ colorlinks=true, linkcolor=black, filecolor=black, urlcolor=black }

\usepackage{lipsum}

\begin{document}
\title[A Fractal Dirac Eigenvalue Problem]{A Fractal Dirac Eigenvalue Problem: \\ Spectral Properties and Numerical Examples} 
\author[F. A. \c Cetinkaya, G. Plott]
{F. Ay\c ca \c Cetinkaya, Gage Plott}  
\date{\today}
\address{F. Ay\c ca \c Cetinkaya and Gage Plott \newline 
 Department of Mathematics, University of Tennessee at Chattanooga, 
 Chattanooga, TN 37403, USA.}
\email{fatmaayca-cetinkaya@utc.edu; gage-plott@utc.edu}

\begin{abstract}
In this paper, we study a Dirac boundary value problem where the operator is considered with a derivative of order $\alpha \in (0,1]$, known as the $F^\alpha$-derivative. We prove some spectral properties of eigenvalues and eigenfunctions and we present numerical examples to demonstrate the practical implications of our approach.
\end{abstract} 

\keywords{Fractal calculus, Fractal derivative, Dirac eigenvalue problem, Eigenvalues, Eigenfunctions}

\maketitle

\section{Introduction}

The word \textit{fractal} derives from the Latin word \textit{fractus} which means \textit{cracked.}  Fractals exhibit unique geometric properties, showcasing a fractal dimension that surpasses their topological dimension. These intricate structures display self-similarity at varying scales, combining repetitive and random processes. The intricate nature of fractals poses challenges for conventional calculus methods in calculating derivatives and integrals. The connection between fractal geometry and natural phenomena, as seen in clouds, mountains, and lightning, underscores a complexity that defies conventional mathematical frameworks, as highlighted in \cite{Mandelbrot1982}. Moreover, in contrast to Euclidean geometry, determining the size of fractals involves non-trivial considerations for measurements like length, surface area, and volume, as discussed in \cite{Falconer2004, Jorgensen2006}.

Fractals, characterized by intricate patterns, showcase self-similarity and often display dimensions that are non-integer and complex \cite{Lapidus2017, Massopust2017}. The middle-third set of Cantor is one of the most well-known examples of fractals \cite{Edgar1998}. Fractals are often too irregular to have any smooth differentiable structure defined on them, and this results in delivering the methods and techniques of ordinary calculus as powerless and inapplicable. The techniques and methods to create calculus on the fractal sets and curves are studied in \cite{Falconer2004, Kigami2001, Strichartz2018}.

In recent work, Parvate and Gangal \cite{Parvate2009, Parvate2011} introduced $F^\alpha$-calculus, a form of Riemannian-like calculus grounded in the fractal subsets of the real line. This calculus is distinguished by its algorithmic simplicity compared to other methods. $F^\alpha$-calculus is a generalization of the ordinary calculus that addresses the cases where standard calculus is inapplicable. In this calculus, an integral of order $\alpha \in (0,1]$ called the $F^\alpha$-integral is defined which makes it possible to integrate functions with fractal support $F$ of dimension $\alpha$.

Furthermore, a derivative of order $\alpha \in (0,1]$, known as the $F^\alpha$-derivative, facilitates the differentiation of functions such as the Cantor staircase function and the Weierstrass function. In contrast to the classical fractional derivative, the $F^\alpha$-derivative aligns the geometrical order of the derivative with the domain of the function's support, lending it a distinct physical interpretation \cite{Golmankhaneh2021elec, Sandev2019, Uchaikin2013}. Notably, the $F^\alpha$-derivative is local, which is in contrast to the non-local nature of the classical fractional derivative. This locality is crucial in physics, where all measurements are inherently local. Moreover, $F^\alpha$-calculus retains much of the simplicity found in ordinary calculus \cite{Golmankhaneh2022}.

Differential equations over fractal domains are often referred to as fractal differential equations. Studies concerning fractal differential equations have been an important area of research in recent years. For instance, in \cite{Golmankhaneh2017lip} Golmankhaneh and Tun\c{c} prove the existence and uniqueness theorems for the linear and non-linear fractal differential equations and they give the fractal Lipschitz condition on the $F^\alpha$-calculus. In \cite{Golmankhaneh2019logistic} Golmankhaneh and Cattini give difference equations on fractal sets and their corresponding fractal differential equations. They define an analogue of the classical Euler method in fractal calculus and they solve fractal differential equations by using this fractal Euler method. In \cite{Golmankhaneh2019sumudu} Golmankhaneh and Tun\c{c} give the analogues of Laplace and Sumudu transforms, which have an important role in control engineering problems, and they solve linear differential equations on Cantor-like sets by utilizing the fractal Sumudu transforms. Fractal differential equations were solved by defining fractal Mellin, Laplace, and Fourier transforms in \cite{Golmankhaneh2023laplace}. Retarded, neutral, and renewal delay differential equations with constant coefficients in the fractal domain are solved through the method of steps and employing Laplace transform in \cite{Golmankhaneh2023delay}. 

In light of the above-given literature, we study the Dirac eigenvalue problem generated by the fractal differential equation 
\begin{equation} \label{1.1}
    \ell^\alpha f := \left\{ \begin{array}{ll}
D_F^\alpha f_2 - p(x) f_1 =\lambda f_1  \\
-D_F^\alpha f_1 + r(x) f_2 =\lambda f_2
\end{array} \right., \quad x \in [0,\pi]
\end{equation}
and the boundary conditions 
\begin{equation} \label{1.2}
    U_1 (f):=f_1 (0)=0,
\end{equation}
\begin{equation} \label{1.3}
    U_2 (f):= f_1 (\pi)=0,
\end{equation}
where $D_F^\alpha$ indicates the $F^\alpha$-derivative introduced in \cite{Parvate2009}, $f=\begin{pmatrix}
f_1\\
f_2
\end{pmatrix}$, $\lambda$ is a real spectral parameter, $p(x), r(x) \in \mathcal{L}_2^\alpha (0,\pi)$ are real valued functions where $\mathcal{L}_2^\alpha (0,\pi)$ is the space of square $F^\alpha$-integrable functions on $[0,\pi]$, i.e.
\begin{equation*}
\int^\pi_0 \lvert f(x)\rvert^2 \dx < \infty
\end{equation*}
holds for $f:[0,\pi] \rightarrow \mathbb{R}$ as $Sch(f)$ is an $\alpha$-perfect set. 

Although numerous studies address various differential equations problems within the fractal calculus framework, there are relatively few that focus on eigenvalue problems in the context of fractal calculus. For instance, in \cite{Cetinkaya2021} \c{C}etinkaya and Golmakhaneh explore a fractal Sturm--Liouville problem, while in \cite{Allahverdiev2024} Allahverdiev and Tuna prove the existence and uniqueness theorem for the solutions of such problems. In this work, we extend the results in \cite{Cetinkaya2021} to the Dirac setting and provide numerical examples to demonstrate the practical implications of our approach. We believe that this work will contribute to further studies related to the eigenvalue problems generated with $F^\alpha$-derivative and their applications.

The structure of the paper is as follows. In Section 2 we introduce a self-adjoint operator and we give some of the virtues
of eigenvalues and vector-valued eigenfunctions. In Section 3 we present numerical examples to illustrate the applicability of our conclusions. In Section 4 we close the paper with some concluding remarks.

\section{Spectral properties}
\label{Sec:3}

In the context of fractal analysis described in \cite{Golmankhaneh2022}, an inner product in the Hilbert space $\mathcal{L}^\alpha_2 (0,\pi)$ can be defined by 
\[
\langle f,g \rangle =\int^\pi_0 \big(f_1(x)g_1(x)+f_2 (x)g_2 (x)\big) \dx
\]
where $f=(f_1, f_2)^T \in \mathcal{L}^\alpha_2 (0,\pi)$ and $g=(g_1, g_2)^T \in \mathcal{L}^\alpha_2 (0,\pi)$.

\begin{theorem}
The operator $\ell^\alpha$ is self-adjoint in $\mathcal{L}^\alpha_2 (0,\pi)$.
\end{theorem}

\begin{proof}
Let $f$ and $g$ be the solutions of the boundary value problem \eqref{1.1}--\eqref{1.3}. Using the definition of the inner product we have
\begin{align*} 
\left\langle \ell^\alpha f, g \right\rangle - \left\langle f , \ell^\alpha g \right\rangle
&= \int^\pi_0 \big(D_F^\alpha f_2 -p(x) f_1\big) g_1 \dx +  \int^\pi_0 \big(-D_F^\alpha f_1 + r(x) f_2\big) g_2 \dx\\ 
&\quad -  \int^\pi_0 f_1 \big(D_F^\alpha g_2 -p(x) g_1\big) \dx -  \int^\pi_0 f_2 \big(-D_F^\alpha g_1 + r(x) g_2\big) \dx\\
&= \int^\pi_0 \big( D_F^\alpha f_2 g_1 + f_2 D_F^\alpha g_1-D_F^\alpha f_1 g_2 - f_1 D_F^\alpha g_2\big)\dx \\
&= \int^\pi_0 D_F^\alpha \big( f_2 g_1 -f_1 g_2\big) \dx.
\end{align*}
Hence
\begin{equation} \label{2.1}
   \left\langle \ell^\alpha  f, g \right\rangle - \left\langle f , \ell^\alpha  g \right\rangle 
   =  \big\{\big(f_2 g_1 - f_1 g_2 \big) \chi_F (x)\big\}_{x=0}^{\pi}.
\end{equation}
Using the boundary conditions \eqref{1.2} and \eqref{1.3}, we see that this term vanishes, so 
\[
 \left\langle \ell^\alpha  f, g \right\rangle - \left\langle f , \ell^\alpha  g \right\rangle=0,
\]
completing the proof.
\end{proof}

Assume that the boundary value problem \eqref{1.1}--\eqref{1.3} has a nontrivial solution 
\[
f(x,\lambda_0)=\begin{pmatrix}
f_1 (x,\lambda_0) \\
f_2 (x,\lambda_0)
\end{pmatrix}
\]
for a certain $\lambda_0$, called an eigenvalue, and the corresponding solution $f(x,\lambda_0)$ is called a vector-valued eigenfunction. 

\begin{lemma}
The vector-valued eigenfunctions $f (x,\lambda_1)$ and $g(x,\lambda_2)$ corresponding to different eigenvalues $\lambda_1 \neq \lambda_2$ are orthogonal.
\end{lemma}
\begin{proof}
Since $f(x,\lambda_1)$ and $g(x,\lambda_2)$ are solutions of \eqref{1.1}, we have
\begin{align*}
D_F^\alpha f_2 (x,\lambda_1)&- p(x) f_1 (x,\lambda_1)=\lambda_1 f_1 (x,\lambda_1),\\
-D_F^\alpha f_1 (x,\lambda_1)&+ r(x)f_2(x,\lambda_1)=\lambda_1 f_2 (x,\lambda_1),\\
D_F^\alpha g_2 (x,\lambda_2)&- p(x)g_1(x,\lambda_2)=\lambda_2 g_1 (x,\lambda_2),\\
-D_F^\alpha g_1(x, \lambda_2) &+r(x)g_2(x,\lambda_2)=\lambda_2 g_2 (x,\lambda_2).
\end{align*}
Multiplying these equations by $g_1(x,\lambda_2)$, $g_2(x,\lambda_2)$, $-f_1(x,\lambda_1)$, and $-f_2(x,\lambda_1)$, respectively, and summing, we get
\begin{align*}
& D_F^\alpha f_2 (x,\lambda_1)\,g_1 (x, \lambda_2) 
+ f_2 (x,\lambda_1)\, D_F^\alpha g_1(x,\lambda_2) \\
&\quad - D_F^\alpha f_1(x,\lambda_1)\,g_2(x,\lambda_2)
- f_1(x,\lambda_1)\,D_F^\alpha g_2(x,\lambda_2)\\
&\quad = (\lambda_1-\lambda_2)
\big(f_1 (x,\lambda_1) g_1 (x,\lambda_2)+f_2 (x,\lambda_1) g_2 (x,\lambda_2)\big).
\end{align*}
That left-hand side is 
\[
\int_{0}^{\pi} D_F^\alpha \big(f_2 (x,\lambda_1) g_1 (x,\lambda_2)-f_1 (x,\lambda_1) g_2 (x,\lambda_2)\big) \,\dx,
\]
which becomes a boundary term that vanishes by conditions \eqref{1.2}, \eqref{1.3}. Hence 
\[
 (\lambda_1-\lambda_2)\int_0^\pi \big[f_1(x,\lambda_1)g_1(x,\lambda_2)
 +f_2(x,\lambda_1)g_2(x,\lambda_2)\big] \dx = 0.
\]
Since $\lambda_1 \neq \lambda_2$, the inner product must be zero, so the eigenfunctions are orthogonal.
\end{proof}

\begin{corollary}
The eigenvalues of the boundary value problem \eqref{1.1}--\eqref{1.3} are real.
\end{corollary}

Let $\varphi(\cdot,\lambda)=\begin{pmatrix}
\varphi_1 (\cdot,\lambda) \\
\varphi_2 (\cdot,\lambda)
\end{pmatrix}$ and $\psi(\cdot,\lambda)=\begin{pmatrix}
\psi_1 (\cdot,\lambda) \\
\psi_2 (\cdot,\lambda)
\end{pmatrix}$ be the solutions of \eqref{1.1} under the initial conditions
\begin{equation} \label{2.2}
\varphi_1 (0,\lambda)=0, \ \varphi_2 (0,\lambda)=1, \ \psi_1 (\pi,\lambda)=0, \ \psi_2 (\pi,\lambda)=1.
\end{equation}
Then
\begin{equation} \label{2.3}
U_1 (\varphi)=U_2 (\psi)=0.
\end{equation}
Denote 
\begin{equation} \label{2.4}
\Delta(\lambda)=\varphi_2 (\cdot, \lambda)\,\psi_1(\cdot, \lambda)-\varphi_1 (\cdot, \lambda)\,\psi_2 (\cdot, \lambda).
\end{equation}
This function $\Delta(\lambda)$ is called the characteristic function of \eqref{1.1}--\eqref{1.3}. One checks that $\Delta(\lambda)$ does not depend on $x$ and is entire in $\lambda$. It has an at most countable set of zeros $\{\lambda_n\}$.
It can be easily seen that the characteristic function does not depend on $x$. Indeed,
\begin{align*} 
D_F^\alpha \Delta (\lambda) &= D_F^\alpha \varphi_2 (\cdot, \lambda) \psi_1 (\cdot, \lambda)+ \varphi_2 (\cdot, \lambda) D_F^\alpha \psi (\cdot, \lambda)-D_F^\alpha \varphi_1 (\cdot, \lambda)\\
&-\varphi(\cdot, \lambda) D_F^\alpha \psi_2 (\cdot, \lambda)\stackrel {\eqref{1.1}}{=} \big(p(\cdot)-\lambda\big) \varphi_1 (\cdot, \lambda)\psi_1 (\cdot, \lambda)\\
&+ \varphi_2 (\cdot, \lambda) \big(r(\cdot)-\lambda\big) \psi_2(\cdot, \lambda)+\psi_2 (\cdot,\lambda) \big(\lambda-r(\cdot)\big)\varphi_2 (\cdot, \lambda)\\
&-\varphi_1 (\cdot, \lambda) \big(\lambda+p(\cdot)\big)\psi_1 (\cdot, \lambda)=0.
\end{align*}
Substituting $x=0$ and $x=\pi$ in \eqref{2.4} and taking \eqref{2.2} into consideration we have
\begin{equation} \label{2.5}
\Delta(\lambda)=U_1 (\psi)=-U_2 (\varphi).
\end{equation}

\begin{theorem}
    The zeros $\{\lambda_n\}$ of the characteristic function coincide with the eigenvalues of the boundary value problem \eqref{1.1}--\eqref{1.3}. The functions $\varphi(x,\lambda_n)$ and $\psi(x,\lambda_n)$ are eigenfunctions, and there exists a sequence $\{\beta_n\}$ such that 
    \begin{equation} \label{2.6}
        \psi(x,\lambda_n)=\beta_n \,\varphi (x,\lambda_n), \ \beta_n \neq 0.
    \end{equation}
\end{theorem}

\begin{proof}
If $\lambda_0$ is a zero of $\Delta(\lambda)$, i.e.\ $\Delta(\lambda_0)=0$, then by construction $\psi(\cdot, \lambda_0)$ is a multiple of $\varphi(\cdot, \lambda_0)$.  Both satisfy the boundary conditions, so $\lambda_0$ is indeed an eigenvalue, and $\psi(\cdot,\lambda_0)$, $\varphi(\cdot,\lambda_0)$ are eigenfunctions.

Conversely, if $\lambda_0$ is an eigenvalue and $f_0$ is a corresponding (nonzero) eigenfunction satisfying \eqref{1.2}--\eqref{1.3}, we can match $f_0$ with one of $\varphi(\cdot,\lambda_0)$ or $\psi(\cdot,\lambda_0)$ (depending on initial conditions), showing $\Delta(\lambda_0)=0$. One also sees each eigenvalue is simple from the geometric point of view.
\end{proof}

We define the weight numbers $\{\alpha_n\}$ of \eqref{1.1}--\eqref{1.3} by
\begin{equation} \label{2.7}
    \alpha_n := \int^\pi_0 \big[\varphi_1^2 (x,\lambda_n)+\varphi_2^2 (x,\lambda_n)\big]\dx. 
\end{equation}

\begin{lemma}
The following relation holds:
\begin{equation} \label{2.8}
    \beta_n \,\alpha_n = D_F^\alpha(\lambda_n),
\end{equation}
where $\beta_n$ are defined by \eqref{2.6} and 
$D_F^\alpha(\lambda)=D_{F,\lambda}^\alpha\Delta(\lambda)$.
\end{lemma}

\begin{proof}
Since $\varphi(x,\lambda_n)$ and $\psi(x,\lambda)$ solve \eqref{1.1}, multiply them in a standard way and integrate over $[0,\pi]$ to get
\[
\big(S_F^\alpha(\lambda_n)-S_F^\alpha(\lambda)\big)\int_0^\pi \big[\varphi_1(x,\lambda_n)\psi_1(x,\lambda)
+ \varphi_2(x,\lambda_n)\psi_2(x,\lambda)\big]\dx
= \Delta(\lambda_n)-\Delta(\lambda).
\]
As $\lambda\to\lambda_n$, the difference quotient leads to
\[
D_F^\alpha(\lambda_n)
= \int_0^\pi \big[\varphi_1(x,\lambda_n)\,\psi_1(x,\lambda_n)
+ \varphi_2(x,\lambda_n)\,\psi_2(x,\lambda_n)\big]\dx.
\]
Substituting \eqref{2.6} into the expression above, we obtain
\[
D_F^\alpha(\lambda_n)
= \beta_n \int_0^\pi \big[\varphi_1^2(x,\lambda_n)\
+ \varphi_2^2(x,\lambda_n)\,\big]\dx.
\]
Now, using \eqref{2.7}, we simplify this to
\[
D_F^\alpha(\lambda_n)
= \beta_n \alpha_n.
\]
Thus, we arrive at \eqref{2.8}.\end{proof}

\begin{corollary}
    The eigenvalues of \eqref{1.1}--\eqref{1.3} are simple from the algebraic point of view, i.e.\ $D_F^\alpha\Delta(\lambda_n) \neq 0$. 
\end{corollary}
\begin{proof}
     Since $\alpha_n \neq 0$, $\beta_n \neq 0$, we get by virtue of \eqref{2.8} that $D_F^\alpha\Delta(\lambda_n) \neq 0$.
\end{proof}
\section{Numerical Examples}

In this section, we present numerical examples to illustrate the applicability of our conclusions. We used the fourth-order classical Runge-Kutta method and the fourth-order fractal Runge-Kutta method as described in \cite{GolmankhanehZayedNum}, whose equations are given in fractal form below. For stiff problems, other numerical techniques may be more suitable. We have plotted solutions to these equations using the Matplotlib Python package. 

\[
y_{n+1}=y_{n}+h_{F}^{\alpha}\left(\frac{k_{n 1}+2 k_{n 2}+2 k_{n 3}+k_{n 4}}{6}\right),
\]
where
\begin{align*} 
& k_{n 1}=f\left(S_{F}^{\alpha}(x), y_{n}\right), \\
& k_{n 2}=f\left(S_{F}^{\alpha}(x)+\frac{1}{2} h_{F}^{\alpha}, y_{n}+\frac{1}{2} h_{F}^{\alpha} k_{n 1}\right), \\
& k_{n 3}=f\left(S_{F}^{\alpha}(x)+\frac{1}{2} h_{F}^{\alpha}, y_{n}+\frac{1}{2} h_{F}^{\alpha} k_{n 2}\right), \\
& k_{n 4}=f\left(S_{F}^{\alpha}(x)+h_{F}^{\alpha}, y_{n}+h_{F}^{\alpha} k_{n 3}\right), \\
& h_{F}^{\alpha} = S_{F}^{\alpha}(x_{n+1}) - S_{F}^{\alpha}(x_n).
\end{align*}

We approximated the integral staircase function $S_{F}^{\alpha}(x)$ by a power law $x^\alpha$ (an approximation valid for certain sets) described in \cite{Parvate2009}. In a problem where the exact structure of the fractal set $F$ is more crucial than the scaling behavior of $\alpha$, one can refer to the implementation in \cite{GolmankhanehZayedNum} of the coarse-grained mass function, $\gamma_{\delta}^{\alpha}(F, a, b)$.

Below are three examples considering \eqref{1.1}--\eqref{1.3} on the real line (with suitable boundary approximations).

\begin{example}
We define $p(x)=\frac{1}{1 + S_{F}^{\alpha}(x)}$, $q(x)=\frac{1}{1+(S_{F}^{\alpha}(x))^2}$, and scaling indices $\alpha = [0.8, 0.9, 1.0]$. The numerical eigenvalues (denoted $\tilde{\lambda}_n$) appear in Table~\ref{table1}. Errors for the case $\alpha=1$ are shown in Table~\ref{table2}.

\begin{table}[H]
\centering
\renewcommand{\arraystretch}{1.42} 
\caption{Numerically computed eigenvalues $\tilde{\lambda}_n$ for various methods and $\alpha$.}
\label{table1}
\begin{tabular}{|c|c|c|c|c|c|}
\hline
\textbf{Method} & $\alpha$ & $\tilde{\lambda}_1$ & $\tilde{\lambda}_2$ & $\tilde{\lambda}_3$ & $\tilde{\lambda}_4$\\
\hline
Classical       & N/A      & 0.347524            & 1.176747            & 2.055970            & 3.020643\\
\hline
Fractal         & 0.8      & 0.413400            & 1.438434            & 2.566015            & -      \\
\hline
Fractal         & 0.9      & 0.378385            & 1.301643            & 2.296227            & -      \\
\hline
Fractal         & 1.0      & 0.347685            & 1.176925            & 2.056040            & 3.020692\\
\hline
\end{tabular}
\end{table}

\begin{table}[H]
\centering
\renewcommand{\arraystretch}{1.42} 
\caption{Magnitude of error between classical and fractal methods for $\alpha=1$.}
\label{table2}
\begin{tabular}{|c|c|c|}
\hline
$\tilde{\lambda}_n$ & Classical & $|\Delta \tilde{\lambda}_n|$ \\
\hline
$\tilde{\lambda}_1$ & 0.347524  & $1.61\times 10^{-4}$\\
\hline
$\tilde{\lambda}_2$ & 1.176747  & $1.78\times 10^{-4}$\\
\hline
$\tilde{\lambda}_3$ & 2.055970  & $7.00\times 10^{-5}$\\
\hline
$\tilde{\lambda}_4$ & 3.020643  & $4.90\times 10^{-5}$\\
\hline
\end{tabular}
\end{table}

The plots of the first eigenfunctions $f_1 = y_1(x)$ and $f_2 = y_2(x)$ for the classical calculus and fractal calculus methods across the three scaling indices are shown below in Figure ~\ref{fig:combinedex1.1}.
\begin{figure}[H]
    \centering
    \caption{}
    \begin{subfigure}[b]{0.45\linewidth}
        \centering
        \includegraphics[width=\linewidth]{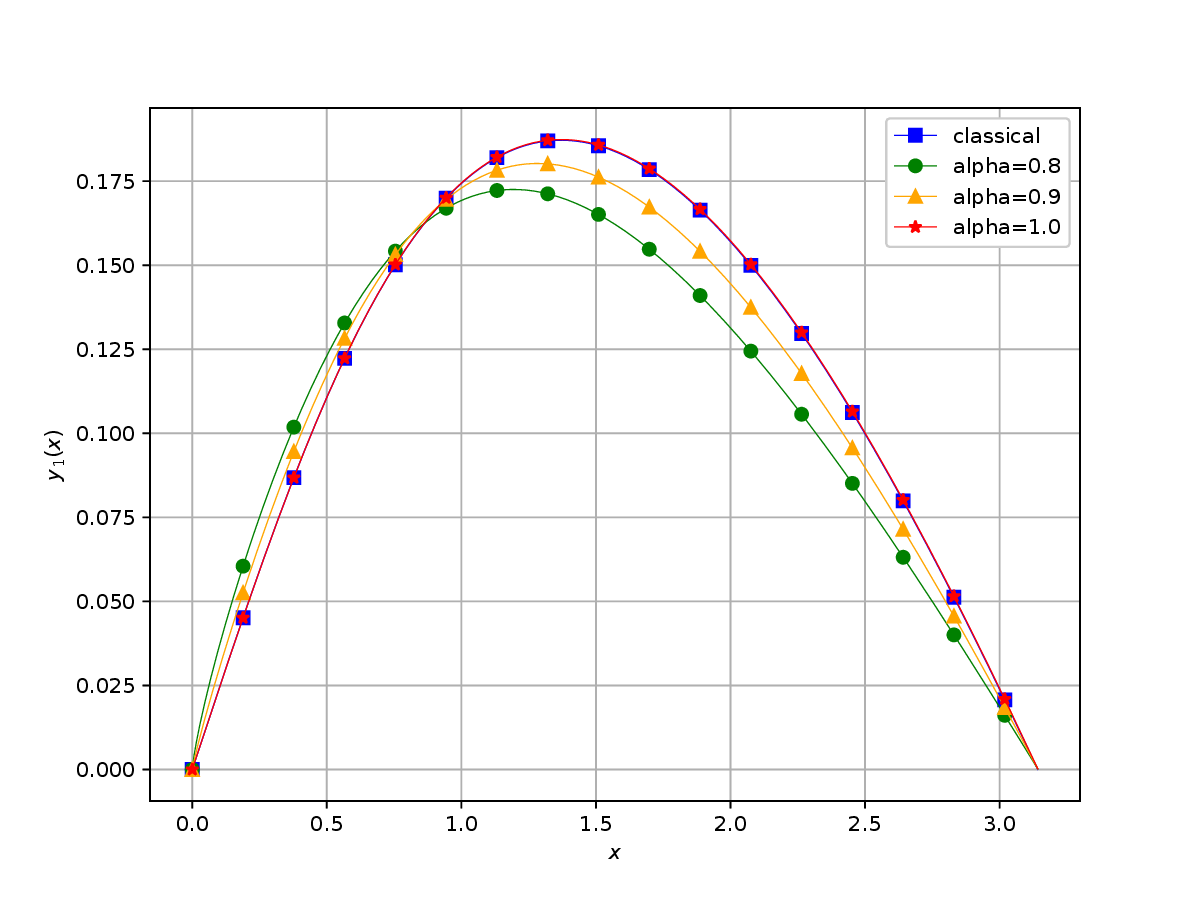}
    \end{subfigure}
    \hfill
    \begin{subfigure}[b]{0.45\linewidth}
        \centering
        \includegraphics[width=\linewidth]{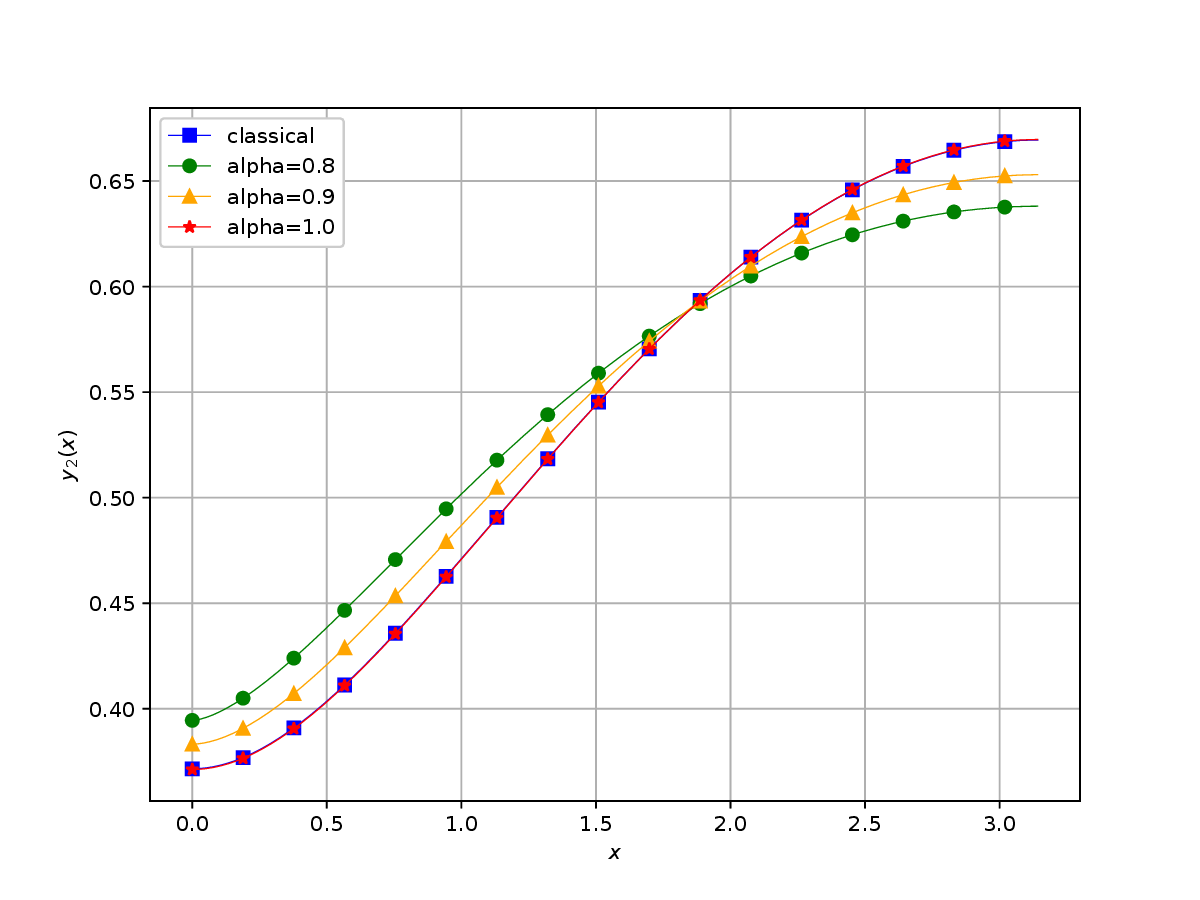}
    \end{subfigure}
    \label{fig:combinedex1.1}
\end{figure}

For easier visual verification that the classical and fractal methods agree when $\alpha = 1$, zoomed-in plots are also provided in Figure ~\ref{fig:combinedex1.2}. 

\begin{figure}[H]
    \centering
    \caption{}
    \begin{subfigure}[b]{0.45\linewidth}
        \centering
        \includegraphics[width=\linewidth]{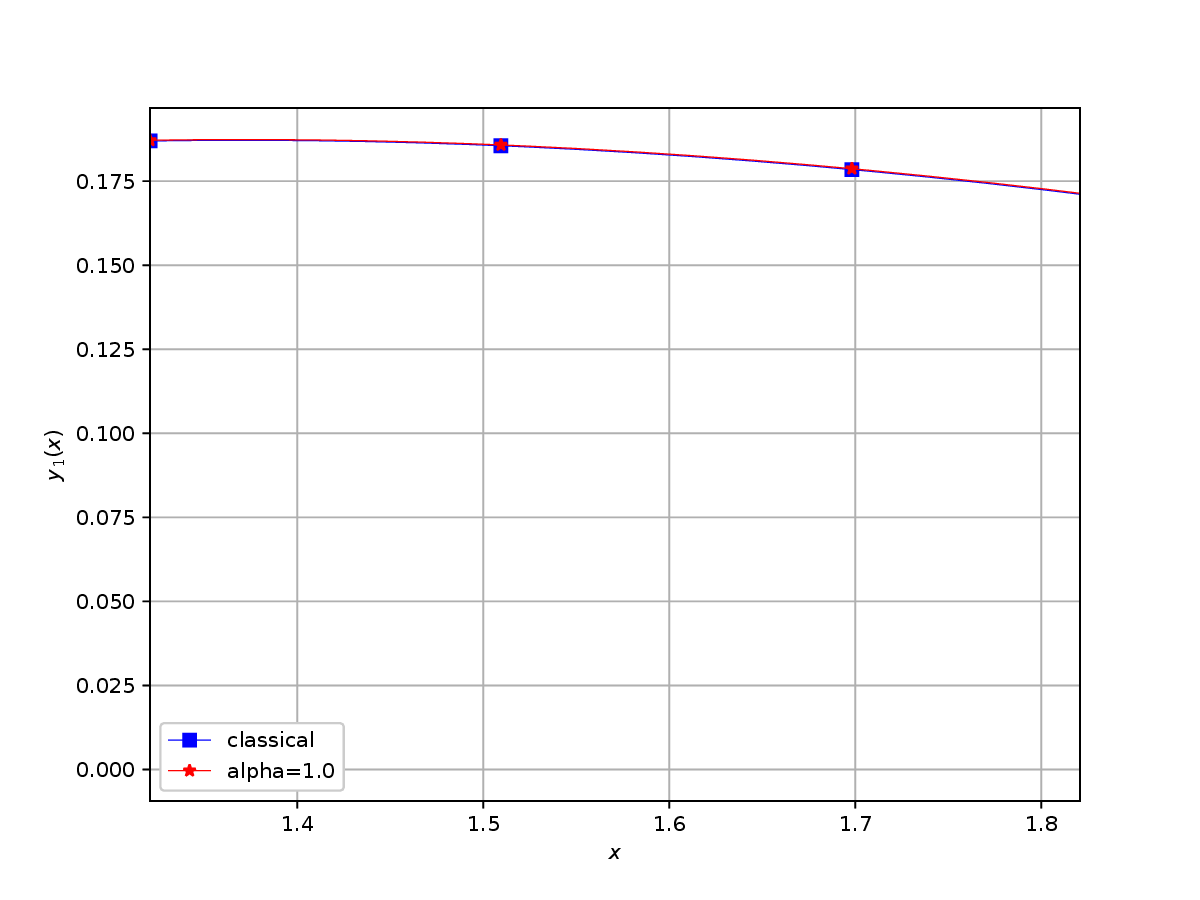}
    \end{subfigure}
    \hfill
    \begin{subfigure}[b]{0.45\linewidth}
        \centering
        \includegraphics[width=\linewidth]{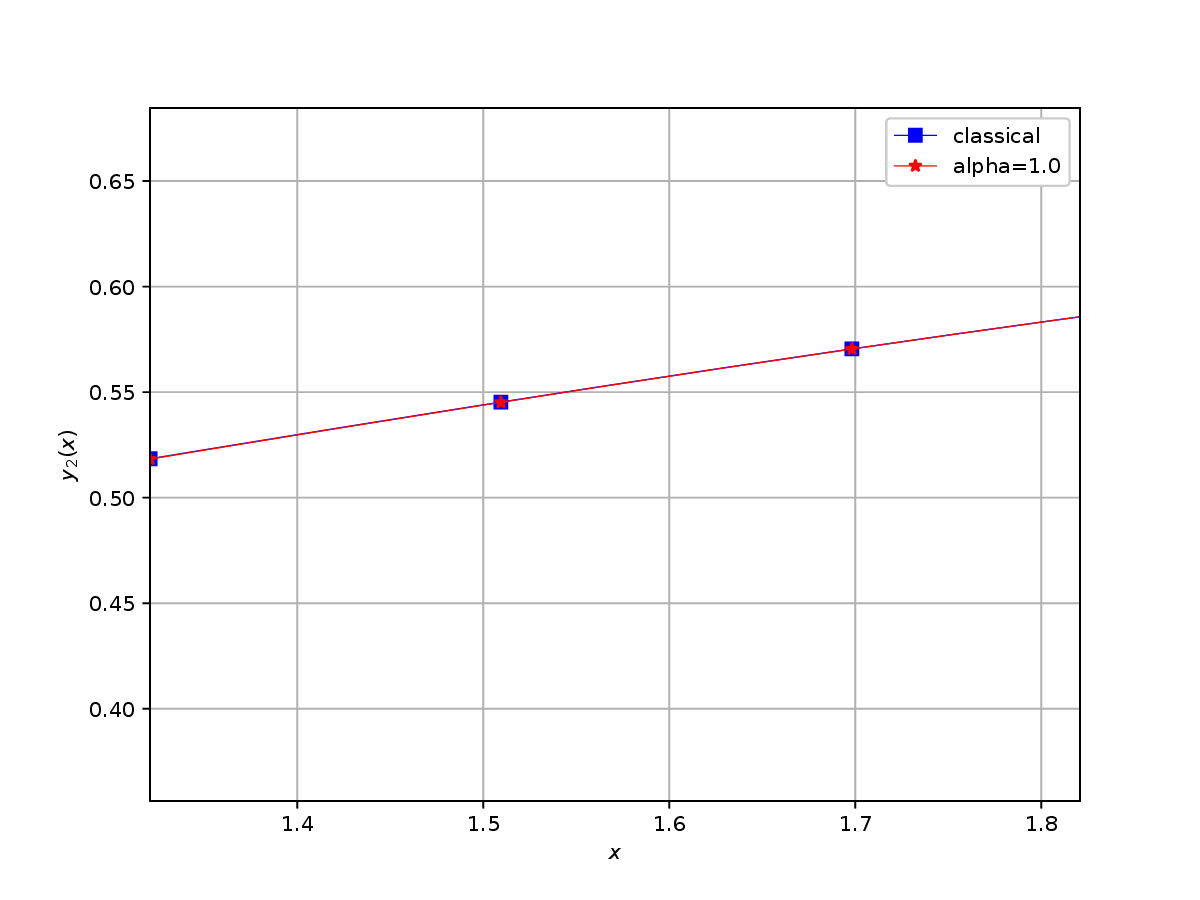}
    \end{subfigure}
    \label{fig:combinedex1.2}
\end{figure}
\end{example}

\begin{example}
Let $p(x) = S_{F}^{\alpha}(x) + 1$, $q(x) = (S_{F}^{\alpha}(x))^2 + 1$, and $\alpha = [0.8, 0.9, 1.0]$. The eigenvalues and errors are shown in Tables~\ref{table3} and \ref{table4}.

\begin{table}[H]
\centering
\renewcommand{\arraystretch}{1.42} 
\caption{Numerically computed eigenvalues $\tilde{\lambda}_n$ for various methods and $\alpha$.}
\label{table3}
\begin{tabular}{|c|c|c|}
\hline
\textbf{Method} & $\alpha$ & $\tilde{\lambda}_1$\\
\hline
Classical       & N/A      & 1.544759\\
\hline
Fractal         & 0.8      & 1.516625\\
\hline
Fractal         & 0.9      & 1.530339\\
\hline
Fractal         & 1.0      & 1.544186\\
\hline
\end{tabular}
\end{table}

\begin{table}[H]
\centering
\renewcommand{\arraystretch}{1.42} 
\caption{Magnitude of error between classical and fractal methods for $\alpha=1$.}
\label{table4}
\begin{tabular}{|c|c|c|}
\hline
$\tilde{\lambda}_n$ & Classical & $|\Delta \tilde{\lambda}_n|$\\
\hline
$\tilde{\lambda}_1$            & 1.544759 & $5.73\times 10^{-4}$\\
\hline
\end{tabular}
\end{table}

The plots of the first eigenfunctions $f_1 = y_1(x)$ and $f_2 = y_2(x)$ for the classical calculus and fractal calculus methods across the three scaling indices are shown below in Figure ~\ref{fig:combinedex2.1}.

\begin{figure}[H]
    \centering
    \caption{}
    \begin{subfigure}[b]{0.45\linewidth}
        \centering
        \includegraphics[width=\linewidth]{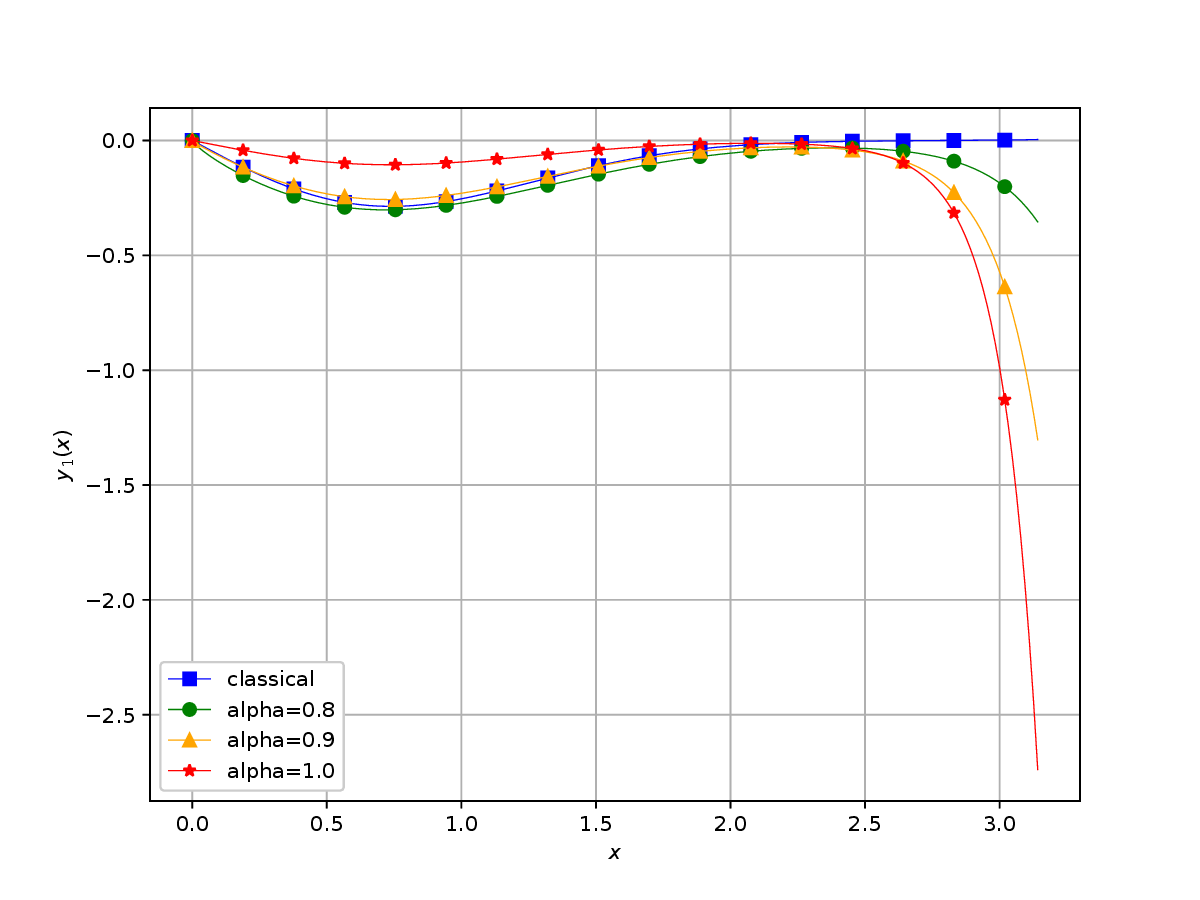}
    \end{subfigure}
    \hfill
    \begin{subfigure}[b]{0.45\linewidth}
        \centering
        \includegraphics[width=\linewidth]{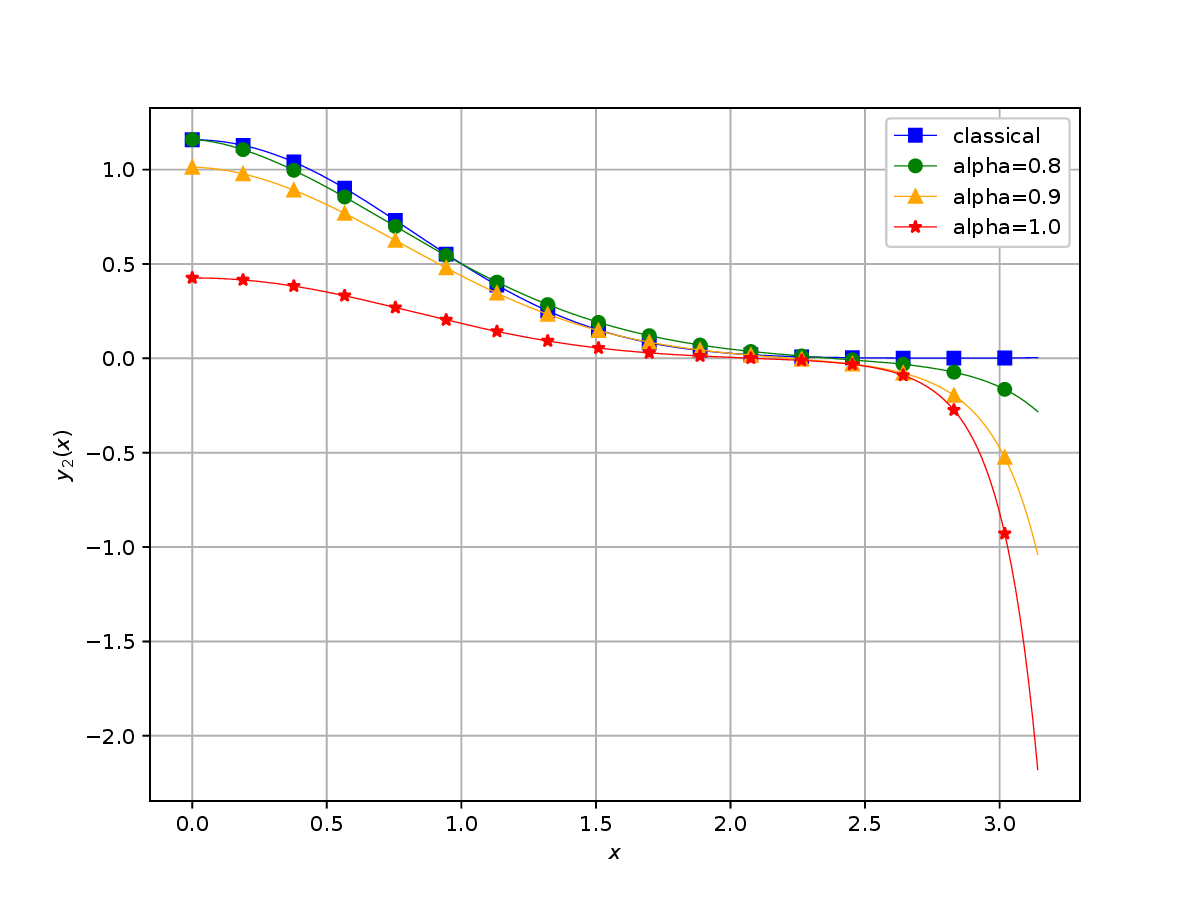}
    \end{subfigure}
    \label{fig:combinedex2.1}
\end{figure}

For easier visual verification that the classical and fractal methods nearly agree when $\alpha = 1$, zoomed-in plots are also provided in Figure ~\ref{fig:combined_2}. 

\begin{figure}[H]
    \centering
    \caption{}
    \begin{subfigure}[b]{0.45\linewidth}
        \centering
        \includegraphics[width=\linewidth]{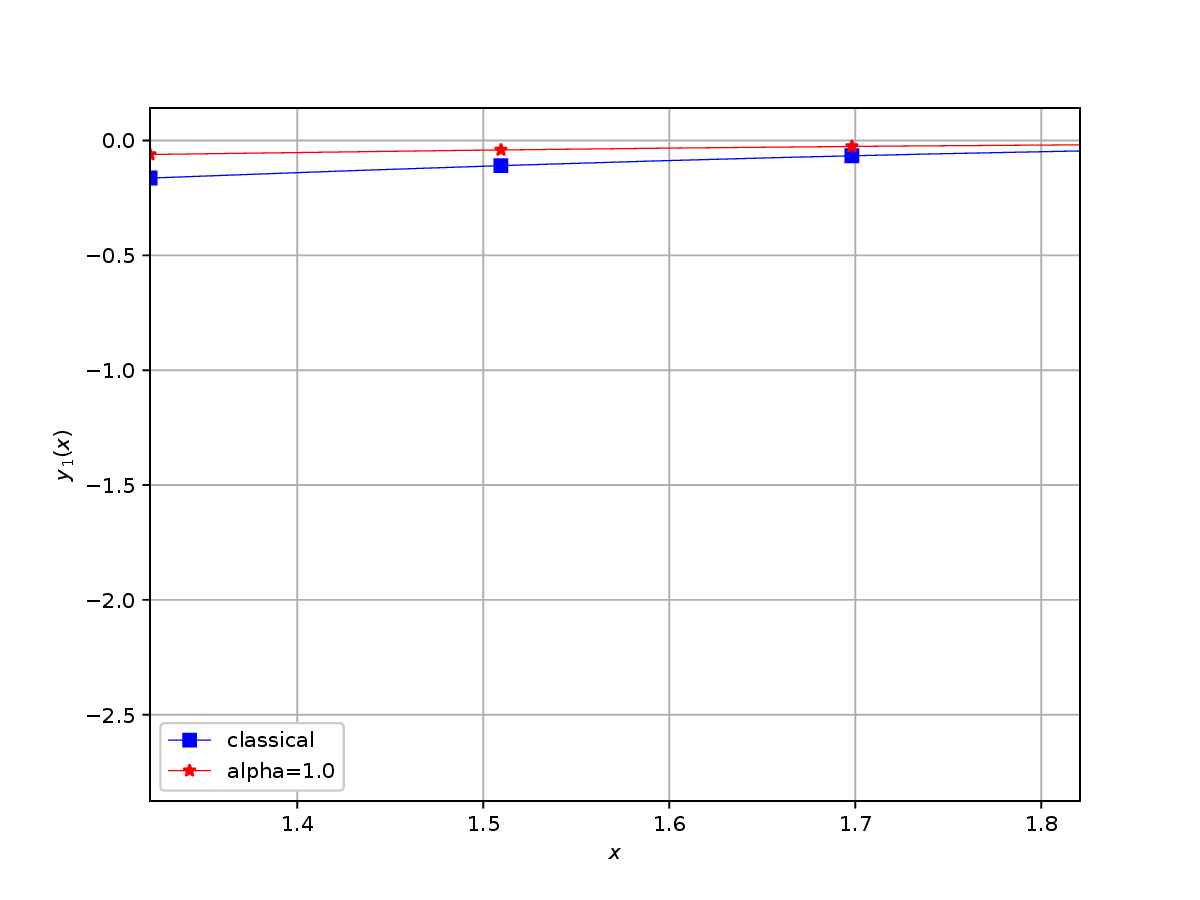}
    \end{subfigure}
    \hfill
    \begin{subfigure}[b]{0.45\linewidth}
        \centering
        \includegraphics[width=\linewidth]{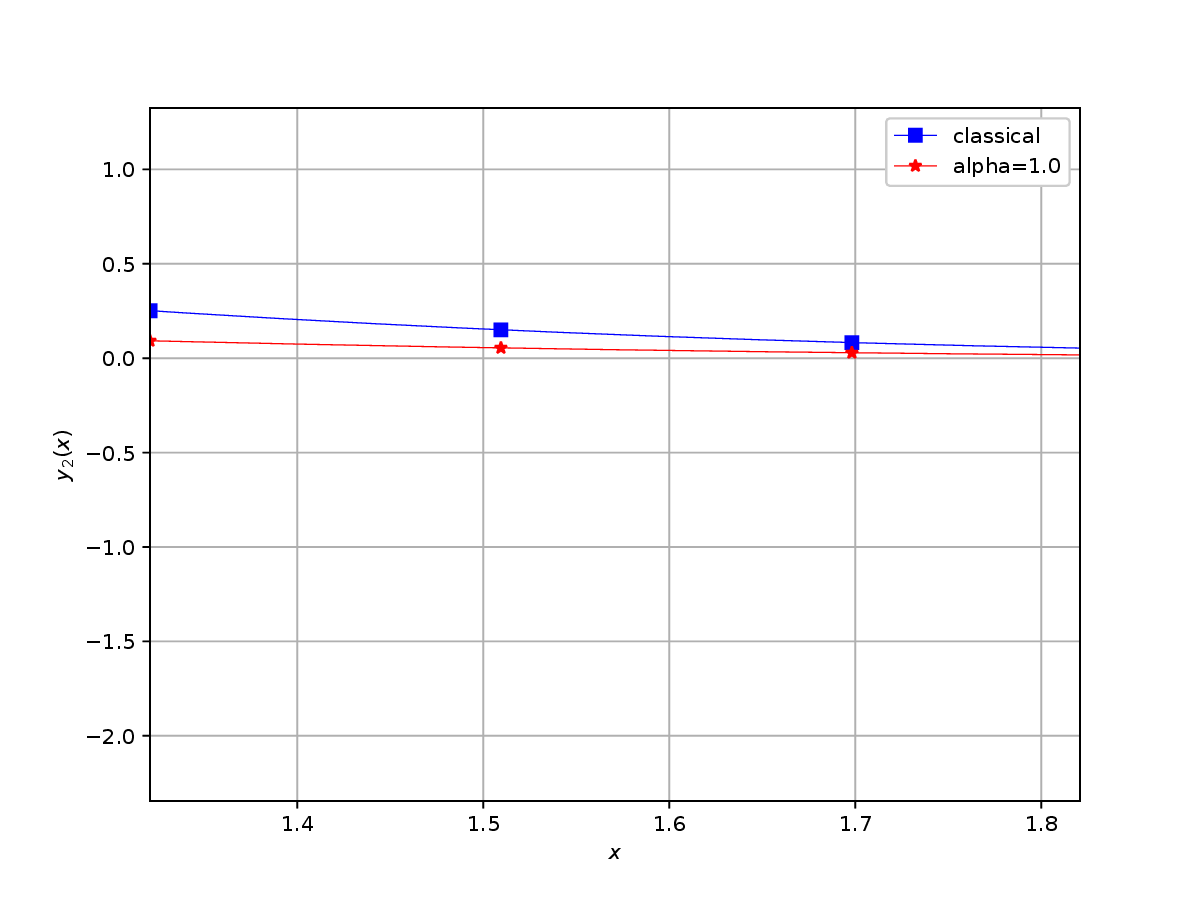}
    \end{subfigure}
    \label{fig:combined_2}
\end{figure} 
\end{example}

\begin{example}
Let $p(x) = e^{S_{F}^{\alpha}(x)}$, $q(x)= e^{-S_{F}^{\alpha}(x)}$, and $\alpha=[0.8, 0.9, 1.0]$. Tables~\ref{table5} and \ref{table6} show the eigenvalues and errors.

\begin{table}[H]
\centering
\renewcommand{\arraystretch}{1.42} 
\caption{Numerically computed eigenvalues $\tilde{\lambda}_n$ for various methods and $\alpha$.}
\label{table5}
\begin{tabular}{|c|c|c|c|c|c|c|c|}
\hline
\textbf{Method} & $\alpha$ & $\tilde{\lambda}_1$ & $\tilde{\lambda}_2$ & $\tilde{\lambda}_3$ & $\tilde{\lambda}_4$ & \(\tilde{\lambda}_5\) & \(\tilde{\lambda}_6\) \\
\hline
Classical & N/A & 0.148677 & 0.458639 & 0.865004 & 1.452401 & 2.170184 & 2.965759\\
\hline
Fractal   & 0.8 & 0.210897 & 0.644622 & 1.301299 & 2.201887 & -        & -\\
\hline
Fractal   & 0.9 & 0.175896 & 0.542309 & 1.057334 & 1.790641 & 2.656072 & -\\
\hline
Fractal   & 1.0 & 0.148792 & 0.458986 & 0.865601 & 1.453235 & 2.171232 & 2.966991\\
\hline
\end{tabular}
\end{table}

\begin{table}[H]
\centering
\renewcommand{\arraystretch}{1.42} 
\caption{Magnitude of error between classical and fractal methods for $\alpha=1$.}
\label{table6}
\begin{tabular}{|c|c|c|}
\hline
$\tilde{\lambda}_n$ & Classical & $|\Delta \tilde{\lambda}_n|$\\
\hline
$\tilde{\lambda}_1$ & 0.148677  & $1.15 \times 10^{-4}$\\
\hline
$\tilde{\lambda}_2$ & 0.458639  & $3.47 \times 10^{-4}$\\
\hline
$\tilde{\lambda}_3$ & 0.865004  & $5.97 \times 10^{-4}$\\
\hline
$\tilde{\lambda}_4$ & 1.452401  & $8.34 \times 10^{-4}$\\
\hline
$\tilde{\lambda}_5$ & 2.170184  & $1.05 \times 10^{-3}$\\
\hline
$\tilde{\lambda}_6$ & 2.965759  & $1.23 \times 10^{-3}$\\
\hline
\end{tabular}
\end{table}

The plots of the first eigenfunctions $f_1 = y_1(x)$ and $f_2 = y_2(x)$ for the classical calculus and fractal calculus methods across the three scaling indices are shown below in Figure ~\ref{fig:combined_ex3}.

\begin{figure}[H]
    \centering
    \caption{}
    \begin{subfigure}[b]{0.45\linewidth}
        \centering
        \includegraphics[width=\linewidth]{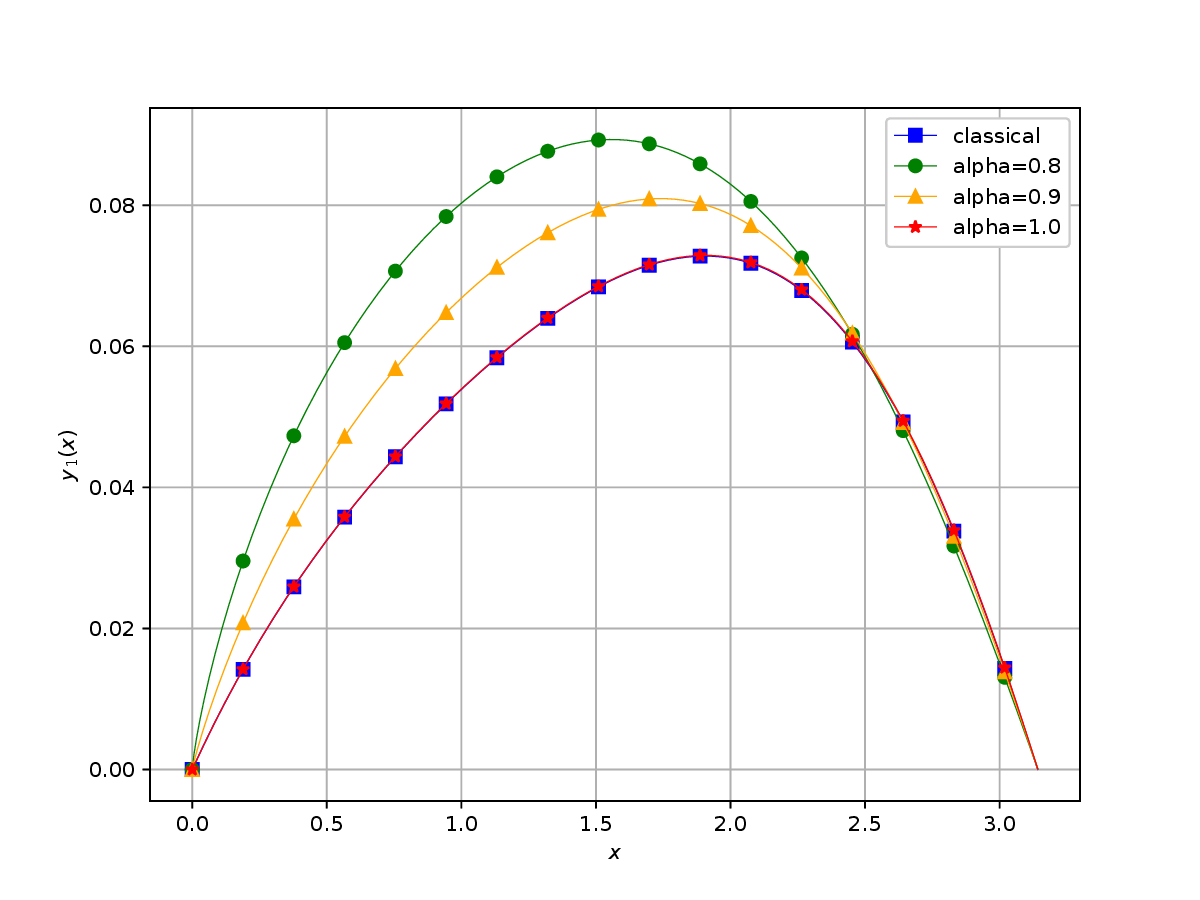}
        \label{fig:fig1_ex3}
    \end{subfigure}
    \hfill
    \begin{subfigure}[b]{0.45\linewidth}
        \centering
        \includegraphics[width=\linewidth]{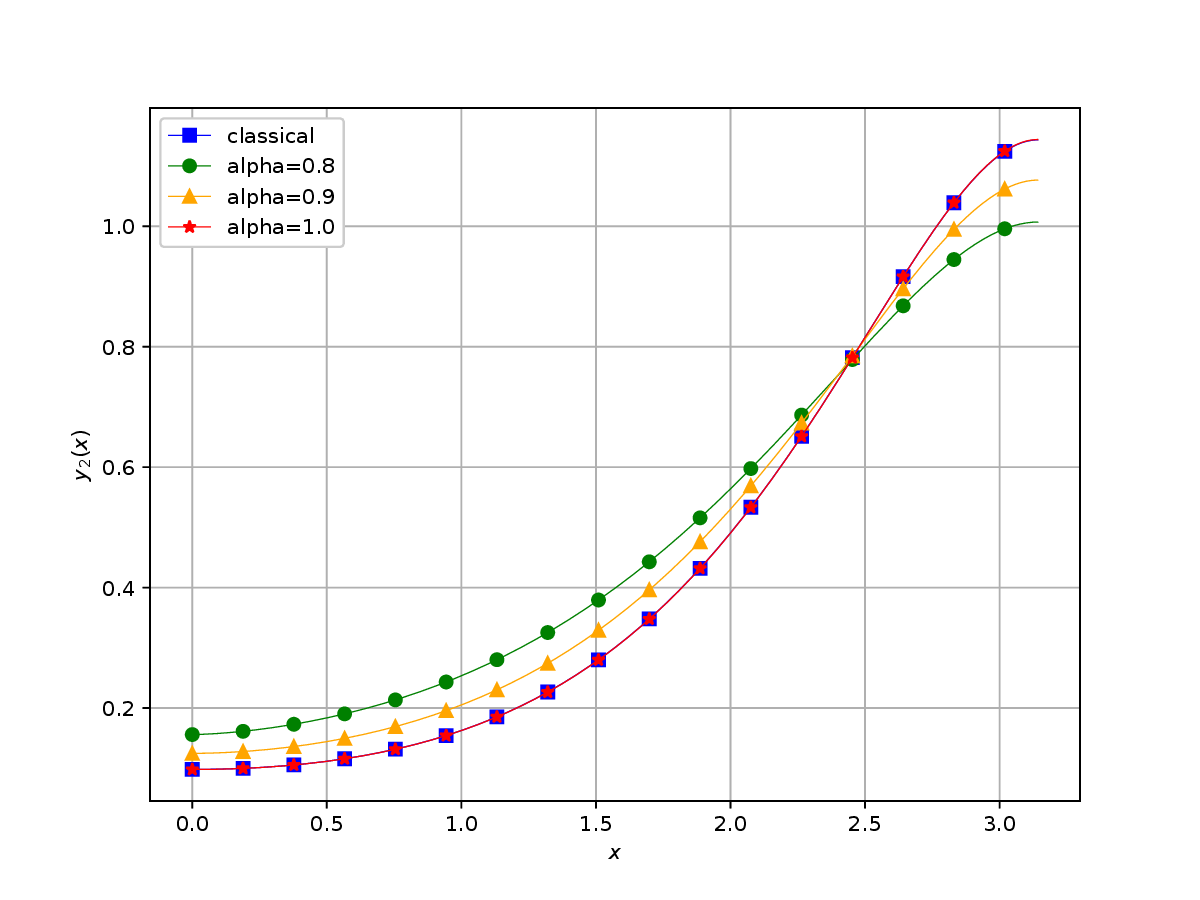}
        \label{fig:fig2_ex3}
    \end{subfigure}
    \label{fig:combined_ex3}
\end{figure}

For easier visual verification that the classical and fractal methods agree when $\alpha = 1$, zoomed-in plots are also provided in Figure ~\ref{fig:combined_ex3_2}.

\begin{figure}[H]
    \centering
    \caption{}
    \begin{subfigure}[b]{0.45\linewidth}
        \centering
        \includegraphics[width=\linewidth]{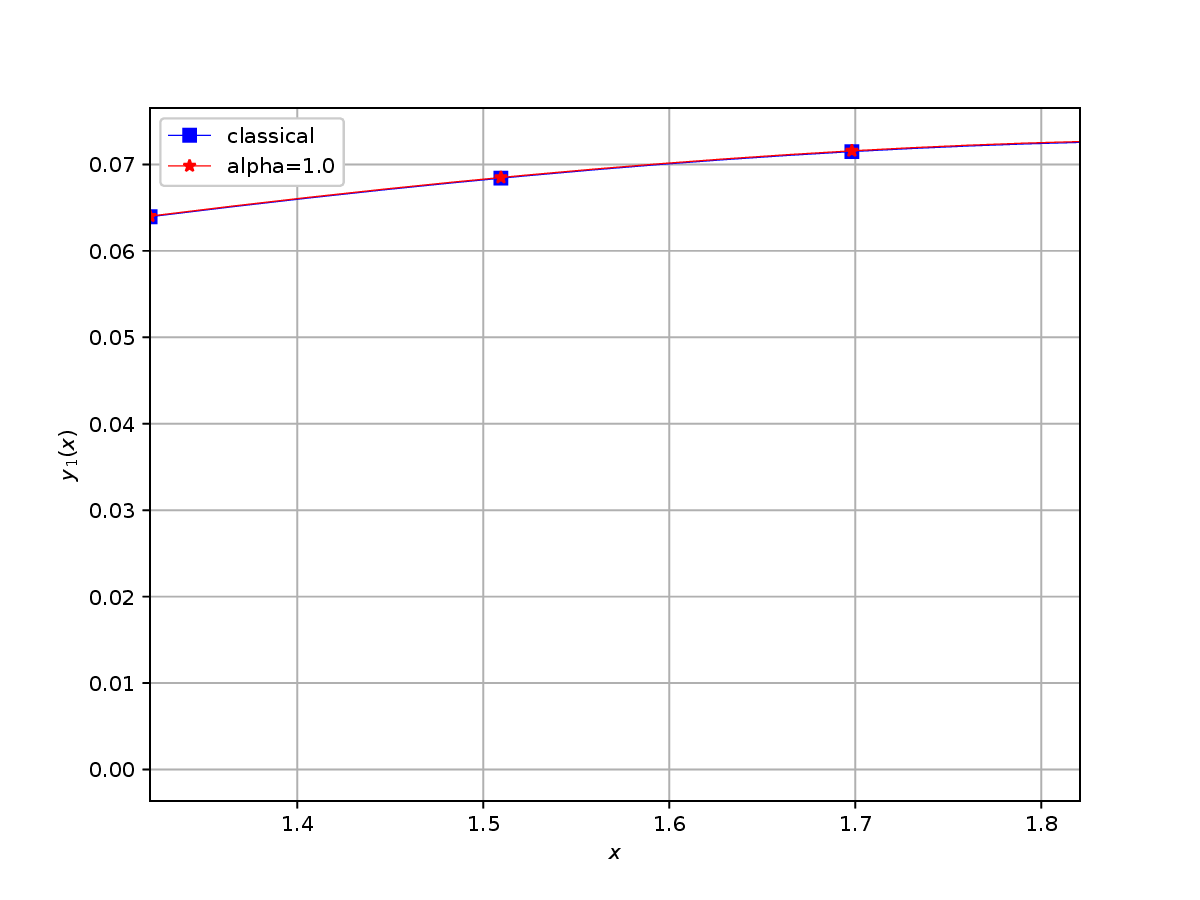}
        \label{fig:fig3_ex3}
    \end{subfigure}
    \hfill
    \begin{subfigure}[b]{0.45\linewidth}
        \centering
        \includegraphics[width=\linewidth]{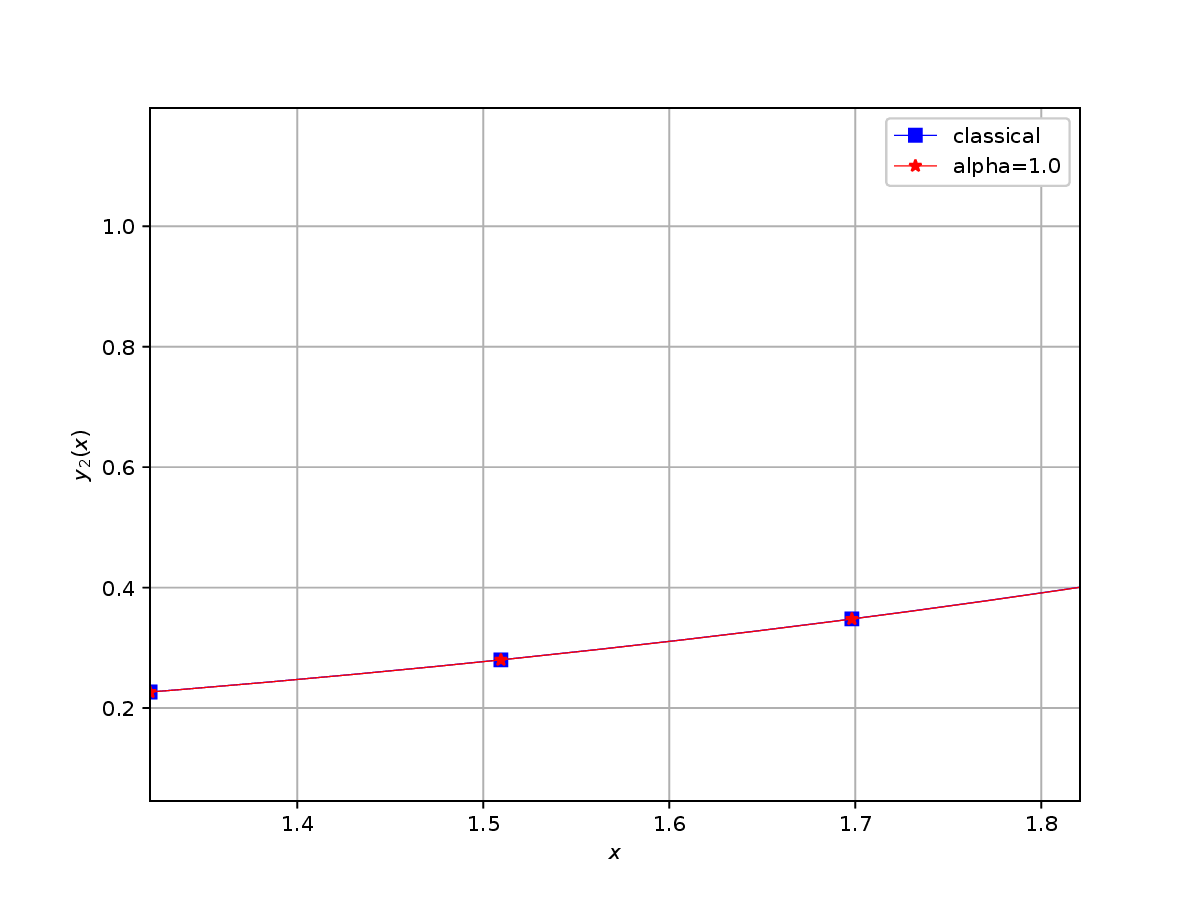}
        \label{fig:fig4_ex3}
    \end{subfigure}
    \label{fig:combined_ex3_2}
\end{figure}   
\end{example}

\section{Conclusion}

In this paper, we have extended the results in \cite{Cetinkaya2021} by introducing and investigating a fractal Dirac problem on a finite interval. We have shown some properties of eigenvalues and eigenfunctions. From the numerical method presented, we see that our theoretical findings are validated. The results are presented for various fractional order derivatives; as the order tends to 1, the eigenvalues recover to an analytical result of the corresponding integer-order problem. The method is simple to implement and can be adapted for other problems.

Some possible future directions include studying the spectral properties of boundary value problems with fractal delay Sturm--Liouville and Dirac equations \cite{Golmankhaneh2024, Golmankhaneh2023delay}, as well as implementing the coarse-grained mass function without approximation in the numerical method to see how the errors improve versus the increased computational expense.

\section*{Declarations}

\subsection*{Funding} The authors declare that there is no funding available for this article.

\subsection*{Conflict of Interest} The authors declare no conflicts of interest.

\subsection*{Ethical Approval} This article does not contain any studies with human participants or animals performed by the authors.

$\,$

$\,$

\bibliographystyle{plain}
\nocite{*}
\bibliography{bibliography}

\end{document}